\documentclass[a4paper,12pt]{article}

\usepackage[british]{babel}
\usepackage{csquotes}
\usepackage[margin=2cm]{geometry}
\usepackage{amssymb,amsmath,amsthm,thmtools,mathtools}
\usepackage[dvipsnames]{xcolor}
\usepackage{todonotes}
\usepackage[hidelinks]{hyperref}
\addto\extrasbritish{%
	\def\Snospace~{\S{}}
}
\usepackage{enumitem}
\usepackage[backend=biber,style=alphabetic,sorting=nyt,backref=true,useprefix=true,maxalphanames=4,minalphanames=3,maxnames=4,minnames=3]{biblatex}
\usepackage{tikz-cd}

\newenvironment{summary}{\begin{quote}\small}{\end{quote}}

\declaretheorem[style=plain,name=Theorem,numberwithin=section]{theorem}
\declaretheorem[style=plain,name=Proposition,sibling=theorem]{proposition}
\declaretheorem[style=plain,name=Lemma,sibling=theorem]{lemma}
\declaretheorem[style=plain,name=Problem,sibling=theorem]{problem}
\declaretheorem[style=plain,name=Corollary,sibling=theorem]{corollary}

\declaretheorem[style=definition,name=Definition,sibling=theorem]{definition}

\declaretheorem[style=definition,name=Conjecture,sibling=theorem]{conjecture}	
\declaretheorem[style=definition,name=Example,sibling=theorem]{example}

\declaretheorem[style=remark,name=Remark,sibling=theorem]{remark}

\declaretheorem[style=definition,name=Question,sibling=theorem]{question}

\newcommand{\from}{\colon}
\newcommand{\leteq}{\coloneqq}
\newcommand{\N}{\mathbb{N}}

\newcommand{\field}{\mathbb{F}}
\newcommand{\Proj}[1]{\mathbb{P}(#1)}
\newcommand{\closure}[1]{\overline{#1}}

\DeclareMathOperator{\PSL}{PSL}
\DeclareMathOperator{\GL}{GL}
\DeclareMathOperator{\PGL}{PGL}

\DeclareMathOperator{\diag}{diag}
\DeclareMathOperator{\rank}{rank}
\DeclareMathOperator{\id}{id}

\DeclareMathOperator{\PStab}{PStab}
\DeclareMathOperator{\Sym}{Sym}
\DeclareMathOperator{\Centre}{Z}

\newcommand{\multiplicativeGroup}{\mathbb{G}_{\text{m}}}

\newcommand{\isom}{\cong}

\newcommand{\cP}{\mathcal{P}}
\newcommand{\cI}{\mathcal{I}}

\newcommand{\transformTo}{\rightsquigarrow}

\newcommand{\mat}[1]{\mathsf{#1}}

\DeclareMathOperator{\coveringNumber}{\textsc{Tcn}}
\DeclareMathOperator{\packingNumber}{\textsc{Tpn}}

\title{Toric decomposition in algebraic groups}
\author{Dávid R. Szabó{\thanks{HUN-REN Alfréd Rényi Institute of Mathematics 
(Hungary, 1053 Budapest, Reáltanoda u. 13-15.). 
The author was supported by the National Research, Development and Innovation Fund -- grant numbers: HIGHLIGHT 153681  
and ADVANCED 153080.
E-mail: {\tt szabo.r.david@gmail.com}}}}
\date{\today}

\begin{document}

\maketitle

\begin{abstract}
    Over an arbitrary field $\field$, we
    construct $n+1$ maximal tori $T_1,\dots,T_{n+1}$ in $\PGL_n(\field)$ so that 
    the product $T_1\dots T_{n+1}$ is almost the whole $\PGL_n(\field)$ and 
    every $g\in T_1\dots T_{n+1}$ can be expressed uniquely as $g=t_1\dots t_{n+1}$ where $t_i\in T_i$. The construction is optimal, as the number of tori with this property attains a general upper bound for connected reductive groups over an algebraically closed field, as well as over finite fields.

    We also show that $n+2$ suitably chosen maximal tori $T_1,\dots,T_{n+2}$ are enough to cover the whole group, i.e. $\PGL_n(\field)=T_1\dots T_{n+2}$, provided $|\field|>n^2$. 
    This is optimal over a finite field and is conjecturally optimal over algebraically closed fields, i.e. the number of such tori is as small as possible.
\end{abstract}

\setcounter{tocdepth}{2}
\tableofcontents

\section{Introduction}

\subsection{Main results}

\begin{definition}
    Let $G$ be a group, $A_1,\dots,A_r$ be subsets of $G$, 
    and consider the \emph{multiplication map} $\mu\from A_1\times \dots \times A_r\mapsto G$ of sets given by $(a_1,\dots,a_r)\mapsto a_1\dots a_r$. 
    The \emph{product} $A_1\dots A_r$ is defined to be the image of $\mu$, i.e. 
    $A_1\dots A_r=\{a_1\dots a_r:a_1\in A_1,\dots,a_r\in A_r\}$. 
\end{definition}

\begin{definition}\label{def:packingCOveringTiling}
    If $G$ is a linear algebraic group, 
    we say  split maximal tori $T_1,\dots,T_r$ \emph{induce a toric packing/covering}, 
    if the multiplication map is injective/surjective.
    Define the \emph{toric packing/covering number} as 
    \begin{align*}
        \packingNumber(G)&\leteq \max\{r\in \N:\text{$T_1,\dots,T_r$ induce a toric packing in $G$}\},\\
        \coveringNumber(G)&\leteq \inf\{r\in \N:\text{$T_1,\dots,T_r$ induce a toric covering in $G$}\}.  
    \end{align*}
\end{definition}
\begin{remark}\label{rem:packingCoveringIntuition}
    If $T_1,\dots,T_r$ induce a toric packing, then 
    every $g\in T_1\dots T_r$ can be uniquely decomposed to the form $g=t_1\dots t_r$ for $t_i\in T_i$. 
    In other words, the left cosets $t_1t_2\dots t_{r-1}T_r$ for $(t_1,\dots,t_{r-1})\in T_1\times \dots \times T_{r-1}$ are pariwise disjoint, i.e. form a packing of the set $G$ using translates of the  torus $T_r$ along $r-1$ different directions.
     On the other hand, if $T_1,\dots,T_r$ induce a toric covering, then $G=T_1\dots, T_r$, explaining the terminology.
\end{remark}

The first statement of this paper gives bounds the following bounds.
\begin{theorem}\label{thm:mainBounds}
    Let $G$ be a nontrivial connected reductive algebraic group over an algebraically closed field. 
    Then 
    \[1\leq \packingNumber(G)\leq \frac{\dim(G)}{\rank(G)}\leq \coveringNumber(G)\leq 2\dim(G)-2\rank(G)+2.\]
\end{theorem}

\begin{remark}
For example, $\packingNumber(\GL_n(\field))=1$, as every split maximal torus contains the (nontrivial) centre of $\GL_n(\field)$, so the product decomposition cannot be unique. This also shows that $\packingNumber(G)$ becomes interesting for a connected reductive group $G$, when $G$ has trivial centre, cf. \autoref{prop:packingBounds}.
\end{remark}

The main result of this paper shows that two other bounds of \autoref{thm:mainBounds} are also (almost) sharp.
\begin{theorem}\label{thm:mainPGL}
    If $\field$ is a field and $n>1$ an integer, 
    then $n+1\leq \packingNumber(\PGL_n(\field))$, 
    and $\coveringNumber(\PGL_n(\field))\leq n+2$ for $|\field|>n^2$.
    In particular, if $\field$ is algebraically closed, then 
    \[\packingNumber(\PGL_n(\field))=\frac{\dim(\PGL_n(\field))}{\rank(\PGL_n(\field))}=n+1\text{,\quad and \quad} 
    \coveringNumber(\PGL_n(\field))\in\{n+1,n+2\}, \]
    analogously, if $\field=\field_q$ is finite, 
    then 
    \[\packingNumber(\PGL_n(\field_q))=n+1\text{,\quad and \quad} 
    \coveringNumber(\PGL_n(\field_q))=n+2 \text{\; for\; $|\field|>n^2$}.\]  
\end{theorem}

Over algebraically closed fields $\field$, the dimension of the product of generic maximal tori of $\PGL_n(\field)$ grows as fast as possible until reaching $\dim(\PGL_n(\field))$.
\begin{corollary}\label{thm:genericToriTransversal}
    Let $\field$ be an algebraically closed field, 
    and let $1\leq s\leq n+1$ be an integer. 
    Then for generic maximal tori $\tilde T_1,\dots,\tilde T_k$ of $\PGL_n(\field)$, 
    we have 
    \[\dim(\tilde T_1\dots \tilde T_s)=\sum_{i=1}^s \dim(\tilde T_i).\]
    In particular, generic elements $g\in \tilde T_1\dots \tilde T_k$ have \emph{finitely many} decompositions of the form $g=t_1\dots t_k$ where $t_i\in\tilde T_i$, 
    a slight relaxation of the packing property from \autoref{def:packingCOveringTiling}.
\end{corollary}

\begin{remark}
    Note that $\frac{\dim(\PGL_n(\field))}{\rank(\PGL_n(\field))}=\frac{n^2-1}{n-1}=n+1$, so \autoref{lem:dim(V1...Vs)} shows the sharpeness of the statement, i.e. that for $s>n+1$ , the conclusion cannot hold.
\end{remark}

\subsection{Motivation}
Product decomposition in algebraic groups is intensively studied. We list some of them without aiming to be comprehensive.

For a connected, reductive algebraic group $G$, the \emph{Bruhat} decomposition is $G=BWB$ where $B$ is a Borel subgroup and $W$ the corresponding Weyl group. 
The big Bruhat cell $BwB$ is a decomposition of an open subset of $G$ in the form of a double coset, but it gives a dense decomposition of the form $B^wB$, the product of two Borel subgroups.
In this fashion, \citeauthor{BRUNDAN1995755}
investigated dense decomposition of the form $X^gX$ for proper reductive subgroups $X$ is an irreducible reductive group $G$ \cite{BRUNDAN1995755}. 
\citeauthor{LiebeckSaxlSeitz} determined all factorisations of a simple algebraic group $G$ of the form $G=XY$ of a  product of two maximal closed subgroups $X$ and $Y$ \cite{LiebeckSaxlSeitz}. 
\citeauthor{guralnick2013products} investigated the product of conjugacy classes, and the product of centralisers in finite and simple algebraic groups \cite{guralnick2013products}.

There are various variants of the covering number. 
The \emph{covering number} $\operatorname{cn}(G)$ of a group $G$ is the minimal $r$ such that $C^r=G$ for every nontrivial conjugacy class $C$ of $G$. 
\citeauthor{LEV199660} showed that $\operatorname{cn}(\operatorname{PSL}_n(\field))=n$ for finite and infnite fields $\field$ \cite{LEV199660}. 
The \emph{covering number} $\operatorname{cn}(G,C)$ of a non-central conjugacy class $C$ of a simple algebraic group $G$ (over an algebraically closed field)  is the minimal $r$ such that $G=C^r$. 
\citeauthor{Liebeck2023CoveringNumbers} showed that $\operatorname{cn}(G,C)\leq 120 \frac{\dim(G)}{\dim(C)}$ \citeauthor{Liebeck2023CoveringNumbers}.
\citeauthor{GordeevSaxl2002a} defined the \emph{extended covering number} $\operatorname{ecn}(G)$ (to be the smallest integer $r$ such that the product $C_1C_2\dots C_r=G$ whenever $C_1,\dots,C_r$ 
are conjugacy classes of $G$ not contained in any proper normal subgroup of $G$) and showed that $\operatorname{ecn}(G)\leq C \rank(G)$ for Chevallley groups $G$ \cite{GordeevSaxl2002a}.

For finite groups, decomposition into various subgroups was studied. \citeauthor{GuralnickMalle2012} showed that every finite simple group is a product of $3$ conjugacy classes \cite{GuralnickMalle2012}. 
\citeauthor{LiebeckPyber2001}
showed that many finite subgroups of $\operatorname{GL}(n,\field_{p^r})$ generated by $p$-elements are actually the product of $25$ of its Sylow $p$-subgroups \cite{LiebeckPyber2001}.
\citeauthor{GaronziMartinoLevyMarotiSimion} improved this to $4$ Sylow $p$-subgroups \cite{GaronziMartinoLevyMarotiSimion}. 
They also studied \emph{conjugate product factorisation} of a group $G$ using a subgroup $A$ (i.e. $G=A^{g_1}\dots A^{g_k}$ for some $g_i\in G$)  
when $a$ is solvable/nilpotent and gave upper bounds to its length $k$ \cite{GaronziMartinoLevyMarotiSimion}. 
\citeauthor{Vavilov2012Unitriangular} considered this problem using conjugates of unitriangular subgroups of Chevalley groups over commutative rings of stable rank $1$  
\cite{Vavilov2012Unitriangular}. 
\citeauthor{smolensky} extended this to some twisted Chevalley groups over finite fields or the field of complex numbers \cite{smolensky}.
Note the toric covering number (\autoref{def:packingCOveringTiling}) is a special case of this conjugate product factorisarion problem as maximal tori are known to be conjugate. 
\citeauthor{Nikolov2007} considered the decomposition of finite quasisimple group of classical type into a product of boundedly many conjugates of central quotients of $\operatorname{SL}$ \cite{Nikolov2007}.
If the multiplication map is bijective, then the product is called a \emph{tiling}. The (non)existence of tiling in finite groups was studied by many authors e.g. \cite{rothaus1966combinatorial}, \cite{fang2026tiling}, \cite{KissMatolcsiMatolcsiSomlai2024}.

\subsection{Outline}
The paper is organised as follows. 
In \autoref{sec:prelim}, we fix some notations and recall useful tools. 
In \autoref{sec:reductive}, we discuss general bounds for irreducible reductive groups and prove \autoref{thm:mainBounds} with dimensional reasoning. 

The main part is \autoref{sec:PGL} where we study $\PGL(V)$ via its regular action on the projective frames in $\Proj{V}$. 
In \autoref{sec:PGLpacking}, considering carefully chosen points, we realise optimal-size toric packings that actually cover an open subset in $\PGL(V)$. Then in \autoref{sec:PGLcovering}, 
we use the Combinatorial Nullstellensatz together with various computations to show that the remaining closed set can actually be reached with the addition of a single maximal torus to complete the proof of \autoref{thm:mainPGL}. 
We prove \autoref{thm:genericToriTransversal} about generic maximal tori in \autoref{sec:PGLgeneric} using a standard argument. 

We conclude the paper with some questions, open problems and conjectures in \autoref{sec:conjectures}.

\subsection{Acknowledgement}
The author is grateful for Endre Szabó for the introduction to the topic and for the many fruitful discussions and ideas. 

\section{Preliminaries}\label{sec:prelim}

\subsection{Notation}
$\N=\{0,1,2,\dots\}$ is the set of natural numbers. 
For $n\in\N$, 
we write $[n]\leteq \{i\in \N:1\leq i\leq n\}$.
$\field$ denotes an arbitrary field (not necessarily algebraically closed). We denote the finite field of order $q$ by $\field_q$.

In a group $G$, we write $h^g=g^{-1}hg$ for the conjugate of $h\in G$ by $g\in G$. For a subgroup $H\leq G$, write $H^g=\{h^g:h\in H\}\leq G$,.
We apply maps from the left.

Let $X$ be a variety. For a constructible subset $A\subseteq X$, we denote its \emph{Zariski-closure} by $\closure{A}$, and write  $\dim(A)=\dim(\closure{A})$ for its \emph{dimension}.

Let $\mat{M}$ be a matrix whose rows and indexed by $R$ and columns by $C$. 
For $I\subseteq R$, $J\subseteq C$, 
we denote by $\mat{M}_{I,J}$ the submatrix of $\mat{M}$ whose rows correspond to $I$, and columns to $J$. 
Write $\mat{M}_J\leteq \mat{M}_{R,J}$.
$\mat{I}_a$ denotes the $a\times a$ identity matrix, 
$\diag(x_1,\dots,x_a)$ the $a\times a$ diagonal matrix, and
$\mat{0}_{a\times b}$ the $a\times b$ matrix having only $0$ entries.
$\GL_n(\field)$ is the group of $n\times n$ invertible matrices over the field $\field$, and $\PGL_n(\field)\leteq \GL_n(\field)/\{\lambda\mat{I}_n:\lambda\in\field^{\times}\}$ is the \emph{projective linear group}. 
We write  $\GL(V)$ and $\PGL(V)$ for the analogous group associated to transformations of the vector space $V$.

\subsection{Tools}

\paragraph{Determinants}
The functoriality of the exterior power gives the following identity, which will be useful to control coefficients of certain polynomials.
\begin{theorem}[Cauchy–Binet formula, {\cite[\S2.9]{shafarevich_remizov_2013}}]\label{thm:Cauchy-Binet}
    Let $1\leq k\leq \min\{a,b,c\}$ be integers, 
    let $I\subseteq [a]$,
    $J\subseteq [c]$ with $|I|=|J|=k$,  
    and let $\mat{A}\in R^{a\times b}$, $\mat{B}\in R^{b\times c}$ be matrices with entries from a commutative ring $R$.
    Then 
    \[\det(\mat{A}\mat{B})_{I,J}=\sum_{K\in [b] : |K|=k} \det (\mat{A}_{I,K}) \det (\mat{B}_{K,J}).\]
\end{theorem}

\paragraph{Nonvanishing locus}
We can control points outside the vanishing set of a finite set of polynomials provided field is large enough.
\begin{theorem}[Combinatorial nullstellensatz, {\cite[Theorem~1.2]{alon1999combinatorial}}]
\label{thm:combinatorialNullstellensatz}
    Let $\field$ be an arbitrary field.  
    Let $f\in\field[X_1,\dots,X_n]$ be a polynomial and 
    $t_1,\dots,t_n\in\N$ be natural numbers  
    such that $\deg(f)=\sum_{i=1}^n t_i$ and the coefficient of $\prod_{i=1}^n X_i^{t_i}$ in $f$ is nonzero. 
    Let $S_i\subseteq \field$ be subsets with $|S_i|>t_i$ for every $1\leq i\leq n$. 
    Then there is $x\in \prod_{i=1}^n S_i$ such that $f(x)\neq 0$.
    
\end{theorem}

\paragraph{Linear algebraic groups}

Let $\field$ be an algebraically closed field. 
Let $\multiplicativeGroup(\field)$ be the multiplicative algebraic group that is isomorphic to $\field^\times$ as groups.
A \emph{torus} in $G$ is a closed subgroup isomorphic to $\multiplicativeGroup(\field)^s$ for some $s\in \N$. 
A \emph{maximal} torus is a torus maximal with respect to containment.
An algebraic group $G$ is \emph{reductive}, if the largest connected, normal, unipotent subgroup is trivial.

We will use the following facts from \cite[\S6,~\S8,~\S14]{Malle_Testerman_2011}.
In an irreducible algebraic group $G$, every maximal tori are conjugate, 
and $\rank(G)$ is defined to be their common dimension. %cite[Corollary~6.5]{Malle_Testerman_2011}.
If $G$ is irreducible and reductive, then 
the union of all maximal tori is dense in $G$, %\cite[Corollary~ 6.11, Corollary~14.10]{Malle_Testerman_2011} 
and every maximal torus $T$ is its own centraliser in $G$, %\cite[Corollary 8.13]{{Malle_Testerman_2011}}
in particular the ${\Centre(G)}$ is contained in every maximal torus of $G$. %\cite[Proposition~6.20]{Malle_Testerman_2011}
,
.
For further details about algebraic groups, the reader is referred to 
\cite{Malle_Testerman_2011} and \cite{Humphreys1975}.

\paragraph{Morphism of varieties}
We will use the following standard fundamental results about morphism of varieties.
\begin{theorem}[Fibre dimension {\cite[\S I.8~Theorems~2-3.]{mumford1999redbook}}]
\label{thm:fibreDim}
Let $f\from X\to Y$ be a morphism of algebraic varieties (over an algebraically closed field). 
For every $y\in f(X)$ and every irreducible component $Z$ of $f^{-1}(y)$,  
\[\dim(X)-\dim(f(X))\leq \dim (Z).
\]
Moreover, there is nonempty open set $U$ in $\closure{f(X)}$ with 
$U\subseteq f(X)$  where equality holds, i.e.   for every $u\in U$, we have   
\[\dim(X)-\dim(f(X))=\dim(f^{-1}(u)).\]
\end{theorem}

\begin{corollary}[Chevalley, {\cite[\S I.8~Corollary~2.]{mumford1999redbook}}]\label{thm:Chevalley}
    Let $f\from X\to Y$ be a morphism of varieties (over an algebraically closed field). 
    If $Z\subseteq X$ is constructible, then so is $f(Z)\subseteq Y$.
\end{corollary}

\section{Decompositions in reductive groups}\label{sec:reductive}
\begin{summary}
    In this section, we prove the general bounds of \autoref{thm:mainBounds} for irreducible reductive groups. 
    Our key tool for the dimensional analysis is the fibre dimension theorem, and the fact that the union of all maximal tori is dense. 
    We also prove general dimensional estimates that will be useful for the whole paper.
\end{summary}

\subsection{Toric packings}

\begin{proposition}\label{prop:packingBounds}
    If $G$ is a nontrivial linear algebraic group (over an algebraically closed field), 
    then \[1\leq \packingNumber(G)\leq \frac{\dim(G)}{\rank(G)}.\]
    The lower bound is sharp if $G$ is irreducible reductive nontrivial centre $\Centre(G)$.
\end{proposition}
\begin{remark}
    This shows that for an irreducible reductive group $G$, the quantity 
    $\packingNumber(G/\Centre(G))$ is much more interesting. 
\end{remark}
\begin{proof}
    For the upper bound, suppose $T_1,\dots,T_s$ induce a toric packing in $G$. 
    Now by \autoref{def:packingCOveringTiling}, the multiplication map $\mu\from T_1\times \dots \times T_S\to G$ is injective, thus 
    \autoref{thm:fibreDim} 
    $s\rank(G)=\dim(T_1\times \dots \times T_S)\leq \dim(G)$, 
    thus $s\leq \frac{\dim(G)}{\rank(G)}$. 
    We are done by \autoref{def:packingCOveringTiling}. 

    For the second part, note that every maximal torus contains the centre $\Centre(G)$. 
    If $T_1,\dots,T_S$ induce a toric packing with $s>1$, then the multiplication map $\mu\from T_1\times \dots \times T_s\to G$ is not injective, as $\mu(z,z^{-1},1,\dots,1)=1$ for every $z\in\Centre(G)$.
\end{proof}
\begin{remark}
    In \autoref{prop:constructionPGL},  we show that the upper bound is sharp for $\PGL_n(\field)$ for $n>1$.
\end{remark}

\subsection{Toric coverings}
\begin{lemma}\label{lem:dim(V1...Vs)}
    If $V_1,\dots,V_s$ are irreducible closed subsets of an  algebraic group $G$ (over an algebraically closed field), then 
    $\closure{V_1\dots V_s}$ is irreducible and 
    \[\dim(V_1\dots V_s)\leq \sum_{i=1}^s \dim(V_i)\]
    with equality if and only if generic elements $g\in V_1\dots V_s$ have only finitely many decompositions of the form $g=v_1\dots v_s$ where $v_i\in V_i$.
\end{lemma}
\begin{proof}
    Consider the morphism $\mu\from V_1\times \dots \times V_s\mapsto G$ given by $(v_1,\dots,v_s)\mapsto v_1\dots v_s$. 
    Note that $V\leteq V_1\times \dots \times V_s$ is irreducible (as all $V_i$ are) and that $f(V)=V_1\dots V_s)$, thus $\closure{f(V)}=\closure{V_1\dots V_s}$ is also irreducible. 
    \autoref{thm:fibreDim} applies and gives an nonepmty open set $U\subseteq \closure{f(V)}$ with $U\subseteq f(V)$ such that 
    \[\sum_{i=1}^s \dim(V_i) - \dim(V_1\dots V_s)
    =\dim(V)-\dim(f(V))
    =\dim(f^{-1}(g))\geq 0\]
    for every $g\in U$ as stated.
    Finally, note that in case of equality, $\dim(f^{-1}(g))=0$, thus $f^{-1}(g)=\{(v_1,\dots,v_s)\in V:v_1\dots v_s=g\}$ is finite.
\end{proof}

\begin{lemma}\label{lem:VH=V}
    Let $G$ be an algebraic group (over an algebraically closed field), 
    let $V$ be a closed irreducible subset of $G$, 
    let $H$ be a closed irreducible subgroup of $G$.
    Then $\dim(VH)=\dim(V)$ if and only if $VH=V$. 
    
    In other words, $\dim(VH)>\dim(V)$ if and only if there exist $h\in H$ with $Vh\not\subseteq V$.
\end{lemma}
\begin{proof}
    Suppose $\dim(VH)=\dim(V)$. 
    Since $1\in H$, we have $V\subseteq VH\subseteq \closure{VH}$. Since $V$ and $H$ are irreducible, so is $\closure{VH}$ by \autoref{lem:dim(V1...Vs)}. 
    Thus the assumption $\dim(V)=\dim(\closure{VH})$ gives  
    $V=\closure{VH}$, hence the chain of containment above forces $VH=V$. 
    The other direction is evident.

    For the other statement, note that $1\in H$ implies $V\subseteq VH$. So $\dim(VH)\neq \dim(V)$ is equivalent to $\dim(VH)>\dim(V)$, whereas 
    and $VH\neq V$ is equivalent to the existence of $h\in H$ with $Vh\not\subseteq V$.
\end{proof}

\begin{lemma}\label{lem:torusGrows}
    Let $G$ be a reductive algebraic group (over an algebraically closed field), 
    let $V$ be a closed subset of $G$ with $\dim(V)<\dim(G)$.
    Then there exists a maximal torus $T$ of $G$ with $\dim(VT)>\dim(V)$.
\end{lemma}
\begin{proof}
    Note that if $V_i$ is an irreducible component of $V$ with $\dim(V_i)=\dim(V)$ and $T$ is a maximal torus of $G$ with $\dim(V_iT)>\dim(V_i)$, then 
    $\dim(V)=\dim(V_i)<\dim(V_iT)\leq \dim(VT)$.
    
    So we may assume that $V$ is irreducible.
    Define $X\leteq \{g\in G:V_ig\subseteq V\}$. 
    Note that $X=\{g\in G:\forall v\in V_i\quad vg\in V\}=\bigcap_{v\in V}v^{-1}V$. 
    Since $V$ is closed by assumption, so is $X$. 
    On the other hand, for any $v\in V$, 
    we have $\dim(X)\leq \dim(v^{-1}V)=\dim(V)<\dim(G)$. 

    Let $U\subseteq G$ be the union of all maximal tori in $G$. 
    Since $G$ is reductive, it is known that $\dim(U)=\dim(G)$. 
    Now $\dim(X)<\dim(U)$, so there is $t\in U\setminus X$, 
    i.e. there exists a maximal torus $T$ containing $t$ such that $Vt\not\subseteq V$. 
    Since $T$ is a closed irreducible subgroup of $G$, \autoref{lem:VH=V} applies with $H\leteq T$ and shows that $\dim(VT)>\dim(V)$ as stated.
\end{proof}

\begin{lemma}\label{lem:UU=G}
    If $U$ is a nonempty open set in an irreducible algebraic group (over an algebraically closed field), 
    then $UU=G$.
\end{lemma}
\begin{proof}
    Pick $g\in G$. 
    Since $G$ is irreducible, $U$ is dense in $G$. 
    The set $U_g\leteq \{gu^{-1}:u\in U\}\isom U$ is nonempty open in $G$,
    thus $U\cap U_g\neq\emptyset$. 
    Hence there exist $u_1,u_2\in U$ with $u_1=gu_2^{-1}$, i.e. $g=u_1u_2$.
\end{proof}

\begin{proposition}\label{prop:TCNbounds}
    If $G$ is a nontrivial irreducible reductive algebraic group (over an algebraically closed field), 
    then 
    \[\frac{\dim(G)}{\rank(G)}\leq \coveringNumber(G)\leq 2(\dim(G)-\rank(G)+1)\leq 2\dim(G).\]
\end{proposition}
\begin{proof}
    Suppose $T_1\dots T_s=G$ for some maximal tori $T_1,\dots,T_s$ of $G$. 
    Then \autoref{lem:dim(V1...Vs)} shows that $\dim(G)=\dim(T_1\dots T_s)\leq \sum_{i=1}^s T_1=s\rank(G)$, so 
    $\dim(G)/\rank(G)\leq s$ and the lower bound follows.

    For the upper bound, we construct a strictly increasing sequence of maximal tori. 
    First, we recursively build a sequence $T_1,\dots,T_d$ of maximal tori such that 
    \[\rank(G)=\dim(T_1)<\dim(T_1T_2)<\dim(T_1T_2T_3)<\dots <\dim(T_1\dots T_d)=\dim(G).\]
    Start with an arbitrary maximal torus $T_1$. 
    Recursively, for every $i\geq 1$, if $\dim(T_1\dots T_i)\leq \dim(G)$, 
    then apply \autoref{lem:torusGrows} for $V\leteq \closure{T_1\dots T_{k-1}}$ to obtain a maximal torus $T_{i+1}$ with $\dim(V)<\dim(VT_{i+1})$. 
    So $\dim(T_1\dots T_i)=\dim(V)<\dim(VT_{i+1})=\dim(T_1\dots T_{i+1})$. 
    Since these dimensions strictly grow, after a finite step, this process will terminate at $i=d$. 
    Evidently, $d\leq \dim(G)-\rank(G)+1$.

    Since $T_1\dots T_d$ is the image of the multiplication map $T_1\times \dots T_d\to G$, it is constructible by \autoref{thm:Chevalley}, 
    thus it contain a nonempty $U$ which is open in $\closure{T_1\dots T_d}$. 
    Since $G$ is irreducible and $\dim(\closure{T_1\dots T_d})=\dim(G)$, we see 
    that $\closure{T_1\dots T_d}=G$, i.e. $U$ is a nonempty open set in $G$. 
    Thus \autoref{lem:UU=G} shows that 
    $T_1\dots T_dT_1\dots T_d=UU=G$. 
    Hence $\coveringNumber(G)\leq 2d\leq 2(\dim(G)-\rank(G)-1)$. 
    The last inequality follows from the fact that $\rank(G)\geq 1$, since $G$ is a nontrivial irreducible reductive group.
\end{proof}

We can now prove one of the main statements.
\begin{proof}[Proof of \autoref{thm:mainBounds}]
    This follows from \autoref{prop:packingBounds} and \autoref{prop:TCNbounds}.
\end{proof}

\section{Decompositions in $\PGL(V)$}
\label{sec:PGL}

\begin{summary}
    In this section, we prove \autoref{thm:mainPGL}, the main result of this paper. 
    We do so in two steps. First, in \autoref{sec:PGLpacking}, we consider an explicit construction for toric packings to cover an open subset of $\PGL(V)$. 
    Then in \autoref{sec:PGLcovering}, we modify this construction to cover the full group using a single maximal torus. 
    Finally, in \autoref{sec:PGLgeneric}, we prove \autoref{thm:genericToriTransversal}.

    The main idea is to study $\PGL(V)$ via its regular action on the projective frames in $\Proj{V}$, using which maximal tori can be characterised as pointwise stabilisers of $\dim(V)$ projectively independent points. Most constructions work over arbitrary fields.
\end{summary}

In this section, the field $\field$ is not assumed to be algebraically closed unless it is explicitly stated. 
We start by fixing the framework of this section.

\begin{definition}
    Let $V$ be an $n$-dimensional vector space over $\field$ with $n\geq 2$. 
    Write $\Proj{V}$ for the projectivisation of $V$, i.e. points $P\in \Proj{V}$ of this $(n-1)$-dimensional projective space are the $1$-dimensional vector subspaces of $V$. 
    A subset $\cP\subseteq \Proj{V}$ is a \emph{projectively independent set},
    if the $1$-dimensional vector subspaces corresponding to elements of $\cP$ are linearly independent. 
    A \emph{projective frame} $(P_0,\dots,P_{n})$ of $\Proj{V}$ is an ordered tuple of 
    points $P_i\in\Proj{V}$, 
    such that any $n$ of these points are projectively independent $\Proj{V}$.
\end{definition}
\begin{remark}\label{rem:frameCoordinates}
    For a projective frame $(P_0,\dots,P_n)$ we may pick representatives $e_k\in V\setminus\{0\}$ of $P_k\in\Proj{V}$ such that $e_1,\dots,e_n$ is a basis of $V$ and $e_0=e_1+\dots +e_n\in V$. If $e'_k$ is any other such choice of vectors, then there exists $\lambda\in\field^{\times}$ so that $e'_k\lambda e_k$ for every $0\leq k\leq n$. Every such choice of vectors gives rise to a projective frame. 
    Every projective frame then gives rise to homogeneous coordinates of points of $P\in \Proj{V}$, i.e. the coordinates of any lift of $P$ to a nonzero $v\in P$ in the basis $e_1,\dots,e_n$ (defined up to nonzero scalar multiples).
\end{remark}

\begin{definition}
    The left action of the general linear group $\GL(V)$ on $V$ induces a left action of the projective linear group $\PGL(V)\leteq \GL(V)/\{\lambda\id_V:\lambda\in\field^\times\}$ on $\Proj{V}$. 
    For a subset $\cP\subseteq \Proj{V}$, write $\PStab(\cP)\leteq \{g\in \PGL(V):\forall P\in \mathcal{P}\quad g(P)=P\}$ for the \emph{pointwise stabiliser} subgroup.
\end{definition}

\begin{remark}\label{rem:regularActionFrame}
    The left induced action $\PGL(V)$ on the set of projective frames of $\Proj{V}$ is regular. 
    More explicitly, for every two (not necessarily different) projective frames $(P_0,\dots,P_{n})$ and  $(Q_0,\dots,Q_{n})$, 
    there is a unique $g\in \PGL(V)$ such that $g(P_i)=Q_i$ for every $i\in\{0,\dots,n\}$. 
    In particular, if $\cP\subseteq\Proj{V}$ contains (points forming) a projective frame of $\Proj{V}$, then  $\PStab(\cP)=1$ is the trivial group.

\end{remark}

\begin{remark}\label{rem:maxTorusPGL}
    If $\cP\subseteq \Proj{V}$ is a projectively independent set of size $|\cP|=\dim(V)=n$, then any (nonzero) lift of the points of $\cP$ to $V$ determines a vector space basis of $V$. 
    In this basis, every element of $\PStab(\cP)$ is represented by diagonal matrices, and all such matrices represent an element of $\PStab(\cP)$. 
    Thus $\PStab(\cP)\isom (\field^{\times})^{n-1}$ is (a split) maximal torus in the algebraic group $\PGL(V)$. 
    Every split maximal torus arises in this way.
\end{remark}

\subsection{Toric packings in $\PGL(V)$}
\label{sec:PGLpacking}
\begin{summary}
    In a fixed projective frame consisting of $n+1$ points, 
    we show that the $n+1$ maximal tori that fix $n$ of these points induce a toric packing in a completely coordinate-free way. 
    We demonstrate the results using concrete matrices and consider concrete cardinality estimates over finite fields. 
\end{summary}

\begin{lemma}[Packing]\label{lem:packing}
    For a vector space $V$, 
    pick subsets $\cP_0,\dots,\cP_s\subseteq \Proj{V}$ such that 
    for every $0\leq k<s$, 
    $\cP_{k+1}\cup \bigcap_{i=0}^{k} \cP_i$ contains (points forming) a projective frame of $\Proj{V}$. 
    Then $\PStab(\cP_s),\dots,\PStab(\cP_0)$ induce a packing in $\PGL(V)$, i.e. the multiplication map $\PStab(\cP_s)\times \dots \times \PStab(\cP_0)\to \PGL(V)$ is injective.
\end{lemma}
\begin{proof}
    We use induction on $s$ with $s=0$ being trivial, so let  $s>0$. 
    Suppose that $t_s\dots t_0=t'_s\dots t'_0$ for $t_k,t'_k\in T_k$ for every $0\leq k\leq s$. 
    Write $T_k\leteq \PStab(\cP_k)$ and $\cI_{s-1}\leteq \bigcap_{k=0}^{s-1}\cP_k$.
    Now \[t_s^{-1}t_s'=(t_{s-1}\dots t_0)(t'_{s-1}\dots t'_0)^{-1}\in  T_s\cap \PStab(\cI_{s-1}),\] because 
    $T_0\dots T_{s-1}\subseteq \PStab(\cI_{s-1})$ by definition and the stabiliser is a subgroup.
    On the other hand, $T_s\cap \PStab(\cI_{s-1})=\PStab(\cP_s\cup \cI_{s-1})=1$ by the assumption and \autoref{rem:regularActionFrame}. 
    Thus $t_s^{-1}t_s'=(t_{s-1}\dots t_0)(t'_{s-1}\dots t'_0)^{-1}=1$, 
    hence $t_s=t'_s$ and $t_{s-1}\dots t_0=t'_{s-1}\dots t'_0$. 
    Finally, by induction, we have $t_i=t'_i$ for $i\leq s-1$ as stated.
\end{proof}
\begin{remark}\label{rem:numerOfTori}
    We want to maximise $s$ in the setup of \autoref{lem:packing} while assuming $|\cP_k|\leq n\leteq \dim(V)$ for every $0\leq k\leq s$. 
    Define $\cI_k\leteq \bigcap_{i=0}^k \cP_i$. 
    Now as the projective frame consists of $n+1$ points, the condition in case $k\neq s$ shows  
    $n+1
    \leq |\cP_{k+1}\cup \cI_k|
    = |\cP_{k+1}|+|\cI_k| - |\cI_{k+1}|
    \leq n + |\cI_k| - |\cI_{k+1}|$, 
    thus 
    $|\cI_{k+1}|\leq |\cI_k|-1$. 
    Since $|\cI_0|=|\cP_0|\leq n$ by assumption, 
    we see that $|\cI_k|\leq n-k$ for every $k$. 
    Hence $0\leq |\cI_s|\leq n-s$ shows the bound $s\leq n$
    with equality if and only if 
    $|\cP_{k}|=n$ 
    and $|\cI_k|=n-k$ for every $0\leq k\leq s$. 
\end{remark}

To simplify the following discussion, we take a specific choice in \autoref{prop:constructionPGL} attaining the optimum of \autoref{rem:numerOfTori}, but many of the conclusions below hold in the general case as well.

\begin{proposition}[Maximal torus packing]\label{prop:constructionPGL}
    Let $V$ be an $ n$-dimensional vector space over  an arbitrary field $\field$. 
    For a projective frame $P\leteq (P_0,\dots,P_{n})$ of $\Proj{V}$, 
    define the split maximal torus $T_k\leteq \PStab(\{P_i:i\neq k\})$ of $\PGL(V)$ for every $0\leq k\leq n$.
    Then $T_n,\dots,T_0$ induce a packing in $\PGL(V)$.    
    In particular, 
    \[\packingNumber(\PGL(V))\geq n+1.\]
\end{proposition}
\begin{proof}
    Since $(P_0,\dots,P_n)$ is a projective frame, 
    the set $\cP_k\leteq \{P_i:i\neq k\}$ is projectively independent,  
    so $T_k=\PStab(\cP_k)$ is indeed a maximal torus in $\PGL(V)$ by \autoref{rem:maxTorusPGL}. 
    For every $k<n$, we have 
    $\cP_{k+1}\cup \bigcap_{i=0}^{k} \cP_i=\{P_0,\dots,P_n\}$, the set of points of the projective frame, 
    hence \autoref{prop:constructionPGL} is applicable and gives the statement.
\end{proof}

\begin{lemma}\label{lem:TCN(PGL(n,q))<=n+1}
    Let $1<n\in\N$ and the prime power $q\in \N$ satisfy $(q-1)\log(q-1)\geq n+1$ (e.g. if $q\geq n+2)$. 
    Then \[\packingNumber(\PGL_n(\field))\leq n+1.\]
\end{lemma}
\begin{remark}\label{rem:q>n/log(n)}
    In fact, the condition asymptotically gives $q\geq 1+\frac{n+1}{\log(n+1)-\log(\log(n+1))}\sim \frac{n}{\log(n)}$. 
    Indeed, denote the solution of $ye^y=x\geq 0$ by $y=W_0(x)$, the Lambert function. Now $(q-1)\log(q-1)\geq n+1$ is equivalent to $q-1\geq \frac{n+1}{W_1(n+1)}$. Using the known asymptotic value of $W_0$ gives the stated bound.
\end{remark}
\begin{proof}
    By contradiction, assume that $T_1,\dots,T_{n+2}$ induce a toric packing in $\PGL_n(\field_q)$. 
    It is well known that $|\PGL_n(\field_q)|=\frac{1}{q-1}\prod_{i=0}^{n-1}(q^n-q^i)=q^{n^2-1}\prod_{j=2}^n (1-q^{-j})< q^{n^2-1}$, 
    and that split maximal tori of $\PGL_n(\field_q)$ are of size $|T_i|=(q-1)^{n-1}$. 
    By assumption, 
    \[\left(\frac{q}{q-1}\right)^{n+1}=\left(1+\frac{1}{q-1}\right)^{n+1}
    \leq \left(1+\frac{1}{q-1}\right)^{(q-1)\log(q-1)}
    < e^{\log(q-1)}=q-1.\]
    Multiplying both sides by $(q-1)^{n+1}$ and then raising to the power $(n-1)$ gives 
    \[|\PGL_n(\field_q)|
    <q^{(n+1)(n-1)}
    <(q-1)^{(n-1)(n+2)}=\prod_{i=1}^{n+2}|T_i|=|T_1\dots T_{n+2}|\leq |\PGL_n(\field_q)|,\]
    where in the end, we used the injectivity of the multiplication map $T_1\times \dots \times T_{n+2}\to \PGL_n(\field_q)$. 
    The two sides of this chain of (in)equalities give a contradiction.
\end{proof}

\begin{example}\label{exmp:PGLmatrix}
Let $e_0,\dots,e_n\in V$ be any choice from \autoref{rem:frameCoordinates} corresponding to the projective frame $P$. 
Now $e_1,\dots,e_n$ is a vector space basis of $V$ which gives an isomorphism $\PGL(V)\isom \PGL_n(\field)$ with the $n\times n$ matrix group.
Under this isomorphism, 
the maximal tori of \autoref{prop:constructionPGL} take the form   
$T_k\isom \mat{T}_k/\{\lambda \mat{I}_n:\lambda \in\field^\times\}$
for $\mat{T}_k\leq\GL_n(\field)$ given by 
$\mat{T}_0\leteq \{\diag(x_1,\dots,x_n):x_i\in\field^\times\}$, and 
\begin{equation}\label{eq:hatTk}
\mat{T}_k\leteq\left\{
\begin{pmatrix}
x_1 & 0 & \cdots & 0 & x_k-x_1 & 0 & \cdots & 0\\
0 & x_2 & \cdots & 0 & x_k-x_2 & 0 & \cdots & 0\\
\vdots & \vdots & \ddots & \vdots & \vdots & \vdots & \ddots & \vdots\\
0 & 0 & \cdots & x_{k-1} & x_k-x_{k-1} & 0 & \cdots & 0\\
0 & 0 & \cdots & 0 & x_k & 0 & \cdots & 0\\
0 & 0 & \cdots & 0 & x_k-x_{k+1} & x_{k+1} & \cdots & 0\\
\vdots & \vdots & \ddots & \vdots & \vdots & \vdots & \ddots & \vdots\\
0 & 0 & \cdots & 0 & x_k-x_n & 0 & \cdots & x_n
\end{pmatrix}
:x_i\in\field^\times
\right\}\end{equation}
 for $1\leq k\leq n$. 
 We will see that the product $\mat{T}_k\dots \mat{T}_0$ consists of matrices with generic elements in the first $k$ columns and on the main diagonal, while $0$ in other entries, see \autoref{rem:stabiliserChain}.
\end{example}

\begin{remark}[Finite analogue of density]
    Consider \autoref{prop:constructionPGL} and assume that $\field=\field_q$ is the $q$-element finite field. 
    We have $|\PGL_n(\field_q)|=q^{n^2-1}-O(q^{n^2-3})$, 
    and every split maximal torus is of size $|T_i|=q^{n-1}-O(q^{n-2})$. 
    Note that the degree of these polynomials in $q$ is precisely the dimension of the corresponding group when the field $\field_q$ is extended to its algebraic closure $\closure{\field_q}$.
    As in the proof of \autoref{lem:TCN(PGL(n,q))<=n+1}, we have 
    $|T_n\dots T_0|=\prod_{i=0}^n |T_i|=(q-1)^{n^2-1}<|\PGL_n(\field_q)|<q^{n^2-1}$, cf. \autoref{lem:TCN>=n+2}. 
    Thus  
    \[\left(1-\frac{1}{q}\right)^{n^2-1} < \frac{|T_n\dots T_0|}{|\PGL_n(\field_q)|} < 1,\]
    so this ratio approaches $1$ as $q\to\infty$. 
    This suggests that when switching to the algebraic closure $\closure{\field_q}$, the product $T_n\dot T_0$ becomes dense in $\PGL_n(\closure{\field_q})$. 
    In fact, we will see in \autoref{lem:Tk...T0-Open} that over an arbitrary algebraically closed field $\field$, the corresponding product $T_n\dots T_0$ is actually open in $\PGL_n(\field)$.
\end{remark}

\subsection{Toric coverings in $\PGL(V)$}
\label{sec:PGLcovering}
\begin{summary}
    We show that the packing construction of \autoref{sec:PGLpacking} actually defined an open subset of $\PGL(V)$. While using twice as many maximal tori, it is easy to reach the remaining closed set, we can do much better. The idea is to move an arbitrary projective frame $Q$ into a fixed one $P$ from \autoref{sec:PGLpacking} using the action of some maximal tori.

    We prove that for a single carefully chosen maximal torus $T$, no orbit in $\Proj{V}$ can be fully contained in the closed set from above. 
    This means that even if we start from the bad closed set, we can `kick  $Q$ off' to the good open set by an appropriate element of $T$. 
    For the technical details, we us the Cauchy--Binet formula for determinants to show that none of the $\dim(V)^2$ polynomials defining the bad closed sets  
    are contant $0$ provided $T$ is chosen suitably. 
    Finally, we use the Combinatorial Nullstellensatz to show that provided the field is large enough, we can avoid all such bad sets simultaneously. 

    We prove \autoref{thm:mainPGL} alongside some concrete estimates in the case of finite fields.
\end{summary}

We describe $T_n\dots T_0$ from \autoref{prop:constructionPGL} more explicitly.

\begin{definition}\label{def:transformTo}
    For two projective frames $Q=(Q_0,\dots,Q_n)$ and $P=(P_0,\dots,P_n)$ of $\PGL(V)$, 
    write $Q\transformTo P$, if 
    $(Q_0,\dots,Q_k,P_{k+1},\dots,P_n)$ is a projective frame for every $0\leq k<n$.

    For a permutation $\sigma\in\Sym(\{0,\dots,n\})$, 
    define the projective frame $P_\sigma\leteq (P_{\sigma(0)},\dots,P_{\sigma(n)})$.
\end{definition}
\begin{remark}\label{rem:transformTo}
    Note that $\transformTo$ is a reflexive binary relation, but is neither symmetric nor transitive. 
    The relation is compatible with the action of $\PGL(V)$, i.e. $Q\transformTo P$ if and only if $g(Q)\transformTo g(P)$. 
    See \autoref{rem:tranformToIntuition} about the choice of the notation.

    $P\cdot \sigma\leteq P_\sigma$ gives a right action on $\Sym(\{0,\dots,n\})$. 
    Note that $P\transformTo P_{\sigma}$ if and only if $\sigma$ is the identity. 
 \end{remark}

\begin{remark}\label{rem:transformToOpen}
    Note that for a fixed $P$, $\{Q:Q\transformTo P\}$ is an open set in the space of projective frames (which is isomorphic to $\PSL(V)$). 
    Indeed,  fix a basis of $V$. 
    For the projective frame $P_0,\dots,P_{n}$, 
    we assign a matrix $\mat{P}\in\field^{n\times (n+1)}$ where the $j$th column of $\mat{P}$ is the coordinate vector of any nonzero representative vector of $P_j$ in $V$, i.e. a choice of the homogeneous coordinates of $P_j$ relative to the fixed basis.
    Similarly, 
    define $\mat{Q}\in \field^{n\times (n+1)}$ for $(Q_0,\dots,Q_{n})$. 
    Now $(Q_0,\dots,Q_k,P_{k+1},\dots,P_n)$ is a projective frame if and only if every $n\times n$ minor of the $n\times (n+1)$ matrix $\begin{pmatrix}
        \mat{Q}_{\{0,\dots,k\}} & \mat{P}_{\{k+1,\dots,n\}}
    \end{pmatrix}$
    is nonzero. (Note that this is independent of the scaling of the homogeneous coordinates.)
    Thus $Q\transformTo P$ if and only if 
    \begin{equation}\label{eq:transformToDet}
        \det\begin{pmatrix}
        \mat{Q}_{\{0,\dots,k\}\setminus\{j\}} & \mat{P}_{\{k+1,\dots,n\}\setminus\{j\}}
    \end{pmatrix}\neq 0
    \end{equation}
    for all $k\in\{0,\dots,n-1\}$, $j\in\{0,\dots,n\}$ 
    where $j$ is the index of the row we delete to obtain an $n\times n$ matrix. 
    The $k=n$ case of $\eqref{eq:transformToDet}$ is equivalent to $Q$ being a projective frame, while the $k=-1$ case is equivalent to $P$ being a projective frame. 
\end{remark}

\begin{definition}
    We say $T_1,\dots,T_r$ induce a \emph{dense toric packing} in $G$, 
    if they induce a toric packing and $\closure{T_1\dots T_r}=G$, i.e. the cosets of $T_1$ from \autoref{rem:packingCoveringIntuition} cover a dense subset of $G$.
\end{definition}
\begin{proposition}[Openness]\label{lem:Tk...T0-Open}
    In the setup of \autoref{prop:constructionPGL}, 
   define the closed subgroup $H_k\leteq \PStab(\{P_i:i>k\})$ of $\PSL(V)$  for every $0\leq k\leq n$. 
    Then we have 
    \[T_k\dots T_0 = \{g\in H_k: g^{-1}(P)\transformTo P\}
    = \{g\in H_k: P\transformTo g(P)\}.\]
    If $\field$ is algebraically closed, then $T_k\dots T_0$ is nonempty (Zariski) open in the closed subgroup $H_k$, i.e. 
    $T_k,\dots,T_0$ induce a dense packing of $H_k$. 
    In particular, $T_n,\dots, T_0$ induce a dense packing in $\PSL(V)$.
\end{proposition}
\begin{remark}\label{rem:tranformToIntuition}
    This explains the intuition behind the $\transformTo$ notation. 
    $Q\transformTo P$ means precisely that the projective frame $Q$ can be transformed to $P$ via the product $T_n\dots T_0$, i.e. 
    that there are (unique) $t_i\in T_i$ for $0\leq i\leq n$ such that $t_n\dots t_0(Q)=P$.
\end{remark}

\begin{proof}
    The second equality is straightforward from \autoref{rem:transformTo}. 
    Let $P\leteq (P_0,\dots,P_n)$.
    We prove the first equality by showing both inclusions. 
    
    For the first containment, pick $g\in H_k$ with $g^{-1}(P)\transformTo P$. 
    We show by induction on $k$ that $g$ has a unique decomposition $g=t_k\dots t_0$ for some $t_i\in T_i$. 
    Indeed, for the $k=0$ base case, note that $H_0=T_0$. 
    So for any $g\in H_0=T_0$, 
    $Q\leteq g^{-1}(P)=(Q_0,P_1,\dots,P_n)$ is a projective frame by \autoref{rem:regularActionFrame}, hence $Q\transformTo P$ as claimed.
    Let $k>0$.     
    Now by the definition of $H_k$, the projective frame $Q\leteq g^{-1}(P)$ is of the form 
    $Q=(Q_0,\dots,Q_k,P_{k+1},\dots,P_n)$ for some $Q_i\in\Proj{V}$. 
    Now $Q\transformTo P$ implies that $Q'\leteq (Q_0,\dots,Q_{k-1},P_k,\dots,P_n)$ is a projective frame. 
    Then the regularity of the action (cf. \autoref{rem:regularActionFrame}) shows the existence of a (unique) $h\in\PGL(V)$ such that $h(Q')=P$ which actually satisfies $h\in H_{k-1}$ by construction, so by induction, we have (unique) $t_i\in T_i$ for $0\leq i\leq k-1$ with $h=t_{k-1}\dots t_0$.  
    Now $h(Q)=(P_0,\dots,P_{k-1},h(Q_k),P_{k+1},\dots,P_n)$, 
    so there is a (unique) $t_k\in T_k=\PStab(\{P_i:i\neq k\})$ such that $t_kh(Q)=P$. 
    On the other hand, $g(Q)=P$ by definition, so the regularity of the action on the frames implies $g=t_kh=t_kt_{k-1}\dots t_0\in T_k\dots T_0$ as required. 

    For the other containment, pick $g\leteq t_k\dots t_0\in T_k\dots T_0$. 
    Since $T_i=\PStab(\{P_j:j\neq i\})$, we have $g\in H_k$ by definition. 
    Now $Q\leteq g^{-1}(P)$ is of the form $(Q_0,\dots,Q_k,P_{k+1},\dots,P_n)$, 
    so $(t_i\dots t_0)^{-1}(P)=(Q_0,\dots,Q_i,P_{i+1},\dots,P_n)$ is a projective frame for every $0\leq i\leq k$, thus $Q\transformTo P$ as required.

    The openness part follows from \autoref{rem:transformToOpen}.
\end{proof}

\begin{example}\label{rem:stabiliserChain}
    The stabiliser subgroup $H_k=\PStab(\{P_i:i>k\})$ of \autoref{lem:Tk...T0-Open} fix $n-k$ points. 
    They form a flag of closed subgroups
    \[1\lneq H_0\lneq H_1\lneq \dots \lneq H_n=\PGL(V).\]
    Under the isomorphism of \autoref{exmp:PGLmatrix}, 
    $H_k\isom \mat{H}_k/\{\lambda \mat{I}_n:\lambda\in\field^{\times}\}$ where 
    \[\mat{H}_k = \left\{\begin{pmatrix}
        \mat{A} & \mat{0}_{k\times (n-k)} \\
        \mat{B} & \diag(x_{k+1},\dots,x_n)
    \end{pmatrix}\in \GL_n(\field):
    \mat{A}\in\GL_k(\field),\mat{B}\in\field^{k\times (n-k)},x_i\in\field^\times \right\}\]
    is the subgroup of $\GL_n(\field)$ with arbitrary element in the first $k$ columns and on the main diagonal, zero everywhere else (such that the matrix is invertible). 
    For the (lift of the) maximal torus $\mat{T}_i\leq \GL_n(\field)$ from \autoref{exmp:PGLmatrix}, 
    \autoref{lem:Tk...T0-Open} shows that $\mat{T}_k\dots \mat{T}_0$ is Zarsiki open in $\mat{H}_k$.
    
    If $\field$ is algebraically closed, then as an algebraic set, 
    $\dim (H_k) = \dim(\mat{H}_k)-1 = %(kn+n-k)-1=
    (k+1)(n-1)$.
    In particular, $\dim(H_{k})-\dim(H_{k-1})=n-1=\rank(\PGL_n(\field))=\dim(T_k)$, so in each step, a maximal torus fits perfectly. 
    In other words, 
    we have 
    \begin{equation}\label{eq:dim(Tk...T0)}
        \dim(T_k\dots T_0)=\dim(H_k)=(k+1)(n-1)=\dim(T_k)+\dots+\dim(T_0)
    \end{equation} 
    for the maximal tori $T_i$ of \autoref{prop:constructionPGL}.
\end{example}
\begin{remark}
    \autoref{thm:genericToriTransversal} shows that  \eqref{eq:dim(Tk...T0)} is the generic behaviour, i.e. that generic maximal tori are \emph{transversal}.
\end{remark}

\begin{lemma}\label{lem:nonzeroDetPolynomial}
    Let $a,b,c\in \N$ with $a=b+c$. 
    Let $\field$ be a field. 
    Let $\mat{A}\in\field^{a\times a}$, $\mat{B}\in \field^{a\times b}$, $\mat{C}\in\field^{a\times c}$ such that
    $\rank(\mat{B})=b$, 
    and     
    any (pairwise different) $b$ columns of $\mat{A}$ together with the columns of $\mat{C}$ form a vector space basis of $\field^a$. 
    For indeterminates $X\leteq (X_1,\dots,X_a)$, let $\mat{D}\leteq \diag(X)\in(\field[X])^{a\times a}$.
    Then the $a\times a$ determinant $f(X)\leteq \det \begin{pmatrix}\mat{A}\mat{D}\mat{B} & \mat{C}\end{pmatrix}\in \field[X]$ is a nonzero multilinear homogeneous  polynomial of degree $b$.
\end{lemma}
\begin{proof}
    \autoref{thm:Cauchy-Binet} applied to the product 
    \[\begin{pmatrix}\mat{A}\mat{D} & \mat{C}\end{pmatrix}\cdot 
    \begin{pmatrix}\mat{B} & \mat{0}_{a\times c}\\\mat{0}_{c\times b} & \mat{I}_c\end{pmatrix}
    = \begin{pmatrix}\mat{A}\mat{D}\mat{B} & \mat{C}\end{pmatrix}
    \]
    of $a\times(a+c)$ and $(a+c)\times (b+c)$ block matrices with $I=[a]$ and $J=[b+c]=[a]$ expresses the $a\times a$ determinant in questions as 
    \begin{align*}
        f(X)= \det\begin{pmatrix}\mat{A}\mat{D}\mat{B}&\mat{C}\end{pmatrix} 
        &= 
    \sum_{K\subseteq[a+c],|K|=a} \det\begin{pmatrix}\mat{A}\mat{D}&\mat{C}\end{pmatrix}_{[a],K}\cdot 
    \det\begin{pmatrix}\mat{B} & \mat{0}_{a\times c}\\\mat{0}_{c\times b} & \mat{I}_c\end{pmatrix}_{K,[a]}
    \\&= \sum_{L\subseteq[a],|L|=b} \det\begin{pmatrix}(\mat{A}\mat{D})_{[a],L}&\mat{C}\end{pmatrix}\cdot 
    \det (\mat{B}_{L,[b]})
    \\&= \sum_{L\subseteq[a],|L|=b} \det\begin{pmatrix}\mat{A}_{[a],L}&\mat{C}\end{pmatrix}\cdot 
    \det (\mat{B}_{L,[b]})\cdot \prod_{l\in L}X_l.
    \end{align*}
    To see the second line of this computation, note that if $a+j\notin K$, then the $b+j$th column of $\mat{M}\leteq \begin{pmatrix}\mat{B} & \mat{0}_{a\times c}\\\mat{0}_{c\times b} & \mat{I}_c\end{pmatrix}_{K,[a]}$ contains only $0$ entries, hence $\det(\mat{M})=0$. 
    So every $K$ giving a potentially nonzero contribution to the sum satisfies $\{a+1,\dots,a+c\}\subseteq K$, i.e. the bottom $c$ rows of $\mat{M}$ must be selected. 
    Write $L\leteq K\setminus \{a+1,\dots,a+c\}\subseteq [a]$ for the selected rows of $\mat{M}$ amongst the top $a$ rows. Note that $|L|=|K|-c=a-c=b$, and that $\det(\mat{M})=\det(\mat{B}_{L,[b]})$. 
    The third line used the fact that $\mat{D}$ is a diagonal matrix.
    
    This shows that $f(X)$ is a multilinear and homogeneous of degree $b$. 
    To see that it is nonzero, note that $\rank(\mat{B})=b$ implies the existence of a nonzero $b\times b$ minor in $\mat{B}$, say $\det(\mat{B}_{L_0,[b]})\neq 0$. 
    By assumption, the $a\times a$ determinant $\det\begin{pmatrix}\mat{A}_{[a],L_0}&\mat{C}\end{pmatrix}\neq 0$,
    therefore the coefficient of $\prod_{l\in L_0}X_l$ in $f(X)$ is nonzero.
\end{proof}

Recall \autoref{def:transformTo}.
\begin{proposition}\label{prop:kicking}
    Let $V$ be an $n$-dimensional vector space over $\field$ with $|\field|>n^2$. 
    Let $P$ and $R=(R_0,\dots,R_n)$ be two projective frames in $\Proj{V}$ such that  
    $R_\sigma\transformTo P$  holds for every $\sigma\in\Sym(\Omega)$ where 
    $\Omega\leteq \{0,\dots,{n-1}\}$.
    Then for every projective frame $Q$ in $\Proj{V}$, there is $t\in\PStab(\{R_i:i\in \Omega\})$ such that $t(Q)\transformTo P$.  
\end{proposition}

\begin{remark}\label{rem:RNC}
    The assumption means that $(R_{\sigma(0)},R_{\sigma(1)},\dots,R_{\sigma(k)},P_{k+1},\dots,P_n)$ is a projective frame for every $k\in\Omega$ and $\sigma\in\Sym(\Omega)$.
    This is satisfied if, for example, the union of the points of $P$ and $R$ are in \emph{general position}, i.e. every $n$ of these points are projectively independent. 
    For $|\field|\geq 2n+1$, we may take a rational normal curve through the points in $P$, and pick another $n+1$ points for $R_0,\dots,R_n$. 
    In fact, taking $R_n=P_n$ is suitable, so it is enough to have $|\field|\geq 2n$.
\end{remark}

\begin{proof}
    Fix a basis of $V$ in which we will represent every object from the statement. 
    Let $\mat{P}\in\field^{n\times (n+1)}$, $\mat{Q}\in\field^{n\times (n+1)}$ and  $\mat{R}\in\GL_n(\field)$ be the matrices corresponding to $P=(P_0,\dots,P_n)$, $Q=(Q_0,\dots,Q_n)$ and  $(R_0,\dots,R_{n-1})$, respectively, as in \autoref{rem:transformToOpen}.

    For every $k\in \Omega$, $j\in\{0,\dots,n\}$, define 
    index sets 
    $J_1\leteq \{0,\dots,k\}\setminus\{j\}$, 
    $J_2\leteq \{k+1,\dots,n\}\setminus\{j\}$, 
    natural numbers 
    $a\leteq n$, 
    $b\leteq |J_1|$, 
    $c\leteq |J_2|$, 
    and matrices 
    $\mat{A}\leteq \mat{R}$, 
    $\mat{B}\leteq \mat{R}^{-1} \mat{Q}_{J_1}$, 
    $\mat{C}\leteq \mat{P}_{J_2}$.
    To apply \autoref{lem:nonzeroDetPolynomial} with this setup, we claim that the conditions are satisfied.
    Indeed, note that $a=n=b+c$. 
    Since $\{Q_j:j\in J_1\}$ is projectively independent by assumption, the columns of 
    $\mat{Q}_{J_1}$ are linearly independent, thus $\rank(\mat{B})=\rank(\mat{Q}_{J_1})=|J_1|=b$ as $\mat{R}\in\GL_n(\field)$. 
    Let $J\subseteq \Omega$ be arbitrary with $|J|=b=|J_1|$. 
    Since $J_1\subseteq \Omega$, there is $\sigma\in\Sym(\Omega)$ with $J=\sigma(J_1)$. 
    By \eqref{eq:transformToDet} for $k=b$, the assumption $R_\sigma\transformTo P$ implies that 
    $\det\begin{pmatrix}
        \mat{A}_{J} & \mat{C}
    \end{pmatrix} = 
    \det\begin{pmatrix}
        \mat{R}_{\sigma(J_1)} & \mat{P}_{J_2}
    \end{pmatrix}\neq 0$,
    thus  any (pairwise different) $b$ columns of $\mat{A}$ together with the columns of $\mat{C}$ form a vector space basis of $\field^a\isom V$, as claimed. 
    Hence \autoref{lem:nonzeroDetPolynomial} is indeed applicable and shows that
    for indeterminates $X=(X_1,\dots,X_n)$ and 
    $\diag(X)\in (\field[X])^{n \times n}$, we have 
    \[f_{k,j}(X)\leteq  
    \det \begin{pmatrix}\mat{A}\diag(X)\mat{B} & \mat{C}\end{pmatrix}
    =\det 
    \begin{pmatrix}
        (\mat{R}\diag(X)\mat{R}^{-1}\mat{Q})_{J_1} & 
        \mat{P}_{J_2}
    \end{pmatrix} \in\field[X]\]
    is a nonzero multilinear homogeneous  polynomial of degree $b$.

    We examine special values of $(k,j)$. 
    First, note that $(k,j)=(0,0)$ if any only if $b=0$, in which case $f_{0,0}(X)=\det(\mat{P}_{\{1,\dots,n\}})\in\field^\times$.
    Second, $(k,j)=(n-1,n)$ if and only if $b=n$, in which case 
    $f_{n-1,n}(X)= \det(\mat{R}\diag(X)\mat{R}^{-1}\mat{Q}_{\Omega})=\det(\mat{Q}_{\Omega})\prod_{i=1}^n X_i$.
    Third, for every $k\in\{0,\dots,n-2\}$, we have $=f_{k+1,k+1}(X)=f_{k,k+1}(X)$ as $J_1=\{0,\dots,k\}$ and $J_2=\{k+2,\dots,n\}$ in both cases.

    For $I\leteq \{(k,j)\in\{0,\Omega\times \{0,\dots,n\}:k\neq j\}$, 
    define \[f(X)\leteq  \prod_{(k,j)\in I} f_{k,j}(X)\in\field[X].\]
    By above, $f(X)$ is a nonzero homogeneous polynomial (of degree $\frac{1}{2}n(n^2+1)$). 
    Being a product of multilinear polynomials, 
    every monomial $\prod_{i=1}^n X_i^{t_i}$ with nonzero coefficient in $f$ satisfies $t_i\leq |I|=n^2<|\field|$ by assumption.  
    Thus \autoref{thm:combinatorialNullstellensatz} is applicable to $f$ with $S_i\leteq \field$, and gives $x\leteq (x_1,\dots,x_n)\in\field^n$ such that $f(x)\neq 0$.
    In particular, $0\neq f_{n-1,n}(x)=\det(\mat{Q}_{\{0,\dots,n-1\}})\prod_{i=1}^n x_i$, thus in fact, $x\in(\field^\times)^n$.

    Define $\mat{T}\leteq \mat{R}\diag(x)\mat{R}^{-1}\in\GL_n(\field)$ and let $t\in\PGL(V)$ be whose matrix in the fixed basis lifts to $\mat{T}$. 
    We claim that $t$ satisfies the statement. 
    First, note that by definition of $\mat{T}$, columns of $\mat{R}$ are eigenvectors of $\mat{T}$, thus $t(R_i)=R_i$ for every $i\in \Omega$ by the definition of $\mat{R}$. Therefore, $t\in\PStab(\{R_i:i\in \Omega\})$. 
    Second, the projective frame $t(Q)$ is represented by $\mat{T}\mat{Q}\in\field^{n\times (n+1)}$. 
    Thus by \eqref{eq:transformToDet}, $t(Q)\transformTo P$ is equivalent to 
    \[f_{k,j}(x)=\det \begin{pmatrix}
        (\mat{R}\diag(x)\mat{R}^{-1}\mat{Q})_{\{0,\dots,k\}\setminus\{j\}} & 
        \mat{P}_{\{k+1,\dots,n\}\setminus\{j\}}
    \end{pmatrix} \neq 0\]
    for every $(k,j)\in \Omega\times j\in\{0,\dots,n\}$. 
    By above, it is enough to require this condition for $(k,j)\in I$. 
    By construction, $0\neq f(x)=\prod_{(k,j)\in I} f_{k,j}(x)$, so we are done.    
\end{proof}
\begin{remark}
    It would be elegant to have a coordinate-free proof of \autoref{prop:kicking}.
\end{remark}

\begin{corollary}\label{cor:TCN(PRL)>=n+2}
    Let $V$ be an $n$-dimensional vector space over a field $\field$ with $|\field|>n^2$. 
    Then 
    \[\coveringNumber(\PGL(V))\leq n+2.\]
\end{corollary}
\begin{proof}
    Let the points $P_0,\dots,P_n, R_0,\dots,R_{n-1}\in \Proj{V}$ be in general position. Such points exists as $|\field|>n^2$, see \autoref{rem:RNC}. 
    Define the maximal tori $T\leteq \PStab(\{R_0,\dots,R_{n-1}\})$ and 
    $T_i\leteq \PStab\{P_j:j\neq i\})$ for $i\in\{0,\dots,n\}$. 
    We prove the statement by showing that $T_n\dots T_1T=\PGL(V)$. 

    Let $Q$ be an arbitrary projective frame in $\Proj{V}$. 
    By \autoref{prop:kicking}, there exists $t\in T$ so that $t(Q)\transformTo P\leteq (P_0,\dots,P_n)$. 
    By \autoref{rem:regularActionFrame}, there is a (unique) $g\in \PGL(V)$ such that $gt(Q)=P$. 
    Now $g^{-1}(P)=t(Q)\transformTo P$, so \autoref{lem:Tk...T0-Open} applied with $k\leteq n$ (so that $H_k=\PStab(\emptyset)=\PGL(V)$) gives (unique) 
    $t_i\in T_i$ for $0\leq i\leq n$ such that $t_n\dots t_0=g$. 
    Therefore $t_n\dots t_1t(Q)=P$. 
    Hence there is an element of the product $T_n\dots T_0 T$ that maps an arbitrary projective frame $Q$ to the fixed projective frame $P$. 
    Hence the uniqueness part of \autoref{rem:regularActionFrame} shows that $T_n\dots T_0 T=\PGL(V)$. 
\end{proof}

Over finite fields, that lower bound $\coveringNumber(\PGL_n(\field))$ is one larger than as is \autoref{prop:TCNbounds} for algebraically closed fields.
\begin{lemma}\label{lem:TCN>=n+2}
    For every $n>1$ and every finite field $\field_q$, we have 
   $\coveringNumber(\PGL_n(\field_q))\geq n+2$. 
\end{lemma}
\begin{proof}
    This follows from a simple counting. 
    Assume $T_1,\dots,T_s$ induce a toric covering. 
    Then $T_1\dots T_s=\PGL_n(\field_q)$, so     
    \begin{align*}
    (q-1)^{(n-1)s}
    &=\prod_{i=1}^s |T_i|
    \geq |T_1\dots T_s|
    = |\PGL_n(\field_q)|
    =\frac{1}{q-1}\prod_{i=0}^{n-1}(q^n-q^i) 
    > (q-1)^{(n-1)(n+1)}
    \end{align*}
    as  $q^n-q^i\geq q^n-q^{n-1}=(q-1)q^{n-1}>(q-1)^n$ holds for every $0\leq i\leq n-1$ and $q\geq 2$.
    This shows that $s>n+1$ as required.
\end{proof}

We are ready to prove the main result of this paper.

\begin{proof}[Proof of \autoref{thm:mainPGL}]
    Let $V$ be an arbitrary $n$-dimensional vector space over $\field$. 
    By picking a basis, we have $\PGL(V)\isom \PGL_n(\field)$. 
    
    The first statement is given by \autoref{prop:constructionPGL} and \autoref{cor:TCN(PRL)>=n+2}.
    
    The next statement about algebraically closed fields follows from the first part and the bounds of \autoref{thm:mainBounds} 
    keeping mind that $\dim(\PGL_n(\field))=n^2-1$ and $\dim(\rank(\PGL_n(\field)))=n-1$.
    
    Similarly, the last statement about finite fields also follows from the first one and the bounds of \autoref{lem:TCN(PGL(n,q))<=n+1} and \autoref{cor:TCN(PRL)>=n+2} as $q>n^2$ implies $(q-1)\log(q-1)\geq n+1$.
\end{proof}

\subsection{Generic tori products}
\label{sec:PGLgeneric}
\begin{summary}
    Since having the product as large dimension as possible is an open condition. Thus, to show the genericness, it is enough to show that the set of suitable maximal tori is not empty. Since the construction of \autoref{sec:PGLpacking} achieves the upper bound, we are done.
\end{summary}

We start with the following standard observation.
\begin{lemma}[Genericness]\label{lem:generic}
    For algebraic varieties $\Gamma,X,Y$ (over an algebraically closed field), a morphism $f\from \Gamma\times X\to Y$, and an integer $d\in \N$, 
    define the set $D\leteq \{\gamma\in\Gamma:\dim(f(\gamma,X))\geq d\}$.
    Then $D$ is empty or 
    there there exists a nonempty open $U$ in $\Gamma$ with $U\subseteq D$.
\end{lemma}
\begin{proof}
    Assume that $D\neq\emptyset$. 
    Define the morphisms of the commutative diagram 
    \[\begin{tikzcd}
        \Gamma\times X \ar[dr,"f"] \ar[r,"F"]
        &  \Gamma\times Y \ar[d,->>,"\pi_2"] \ar[r,->>,"\pi_1"]
        & \Gamma
        \\
        X \ar[u,hook,"\iota_\gamma"] \ar[r, "f_\gamma"]
        & Y
    \end{tikzcd}\]
    by $F\from (\gamma,x)\mapsto (\gamma,f(\gamma,x))$, 
    $\pi_1\from (\gamma,y)\mapsto \gamma)$, 
    $\pi_2\from (\gamma,y)\mapsto y$, 
    and 
    $f_\gamma\from x\mapsto f(\gamma,x)$,  
    $\iota_\gamma\from x\mapsto (\gamma,x)$ for every $\gamma\in\Gamma$.
    Apply \autoref{thm:fibreDim} to $f_\gamma$, and note 
    $f_\gamma^{-1}(y)\isom F^{-1}(\gamma,y)$ (via $\iota_\gamma$) to get 
    \begin{align*}
        D&=
        \{\gamma\in \Gamma:\dim(f_\gamma(X))\geq d\}
        \\&= \{\gamma\in \Gamma:\dim(X)-\min_{y\in f_\gamma(X)}(\dim(f_\gamma^{-1}(y)))\geq d\}
        \\&= \{\gamma\in \Gamma:\exists y\in f_\gamma(X)\quad \dim(F^{-1}(\gamma,x))\leq \dim(X)-d\}
        \\&= \pi_1(V),
    \end{align*}
    where $V\leteq \{(\gamma,y)\in F(\Gamma\times X): \dim(F^{-1}(\gamma,y))\leq \dim(X)-d\}$.
    Let $U_F\subseteq \closure{F(\Gamma\times X)}$ be the nonempty open set with $U_F\subseteq F(\Gamma\times X)$ given by \autoref{thm:fibreDim} when applied to $F$.  
    Now $\dim(F^{-1}(u))=r\leteq \dim(\Gamma\times X)-\dim(F(\Gamma\times X))$ 
    for every $u\in U_F$, 
    whereas $\dim(F^{-1}(\gamma,y))\geq r$ for every $(\gamma,y)\in F(\Gamma\times X)$. 
    Since $D=\pi_1(V)$ is nonempty by assumption, we see that $V\neq\emptyset$. 
    This forces $\dim(X)-d\geq r$, 
    and consequently, $U_F\subseteq V$. 
    Since $F\circ \pi_1$ is surjective, 
    we have $\pi_1(\closure{F(\Gamma\times X)})=\Gamma$. 
    Thus by \autoref{thm:Chevalley}, $\pi_1(U_F)$ is constructuble and hence if $U$ is a locally closed component of $f(U_F)$ of maximal dimension, then $U$ is open in $\Gamma$. 
    Finally $U\subseteq \pi_1(U_F)\subseteq \pi_1(V)=D$ as required.
\end{proof}

We are ready to prove the final main statement of the paper.
\begin{proof}[Proof of \autoref{thm:genericToriTransversal}]
    Let $G\leteq \PGL(V)$, and fix an arbitrary maximal torus $T$ of $G$. 
    Now $\{T^g:g\in G\}$ is the complete set of maximal tori of $G$, since for every irreducible linear algebraic group, all maximal tori are conjugate. 
    For 
    $\Gamma\leteq G^s$, 
    $X\leteq T^s$, and  
    $Y\leteq G$, 
    define the morphism 
    \[f\from \Gamma\times X\to Y,\quad (g_1,\dots,g_s,t_1,\dots,t_s)\mapsto t_1^{g_1}\dots t_s^{g_s}\]
    of varieties, 
    and let $D\leteq \{\gamma\in \Gamma:\dim(f(\gamma,X))\geq d\}$, 
    for $d\leteq s\cdot (n-1)\in \N$.
    Note that for every $\gamma=(g_1,\dots,g_s)\in\Gamma$, we have $f(\gamma,X)=T^{g_1}\dots T^{g_s}$. 
    In particular, for every $(g_1,\dots,g_s)\in D$, we have  $\dim(T^{g_1}\dots T^{g_k})=s\cdot (n-1)=\sum_{i=1}^s \dim(T^{g_i})$.

    Let $T_k,\dots,T_0$ be the maximal tori from \autoref{prop:constructionPGL} for $k=s-1$. 
    Pick $g^*_i\in G$ with $T^{g^*_i}=T_{s-i}$ for $1\leq i\leq s$, 
    and let $\gamma^*\leteq (g^*_1,\dots,g^*_s)\in\Gamma$. 
    Now $\dim(f(\gamma^*,X))=\dim(T_k\dots T_0)=s\cdot (n-1)=d$ by \eqref{eq:dim(Tk...T0)}, hence $\gamma^*\in D$ by definition.
    Therefore \autoref{lem:generic} shows that there is $U_\Gamma\subseteq D$ for some nonempty open set $U_\Gamma$ in $\Gamma$.

    Let $H\leteq N_G(T)$ be the normaliser of $T$ in $G$. This is a closed subgroup, so the set of right cosets $\mathcal{T}\leteq \{Hg:g\in G\}$ is an algebraic variety and it gives a one-to-one parametrisation of maximal tori in $G$. 
    Let $\pi\from G\to \mathcal{T}$ denote the natural projection. 
    Then $\pi^s\from \Gamma\to \mathcal{T}^s$ is a surjective morphism (to the space of $s$-tuples of maximal tori of $G$), so 
    there is a nonempty open set $U$ in $\mathcal{T}^s$ such that 
    $U\subseteq\pi^s(U_\Gamma)$. 
    By above, for any $(\tilde{T}_1,\dots,\tilde{T}_s)\in U$ (i.e. for $s$ generic maximal tori of $G$),  
    we have $\dim(\tilde T_1\dots \tilde T_k)\geq \sum_{i=1}^k \dim(\tilde T_i)$. 
    The inequality also holds the other way by  \autoref{lem:dim(V1...Vs)} when applied to $V_i\leteq \tilde T_i$, and thus the statement follows.
\end{proof}

\section{Open problems and conjectures}
\label{sec:conjectures}

In \autoref{sec:PGL}, we analysed $\packingNumber(\PSL_n(\field))$ and $\coveringNumber(\PSL_n(\field))$. Natural candidates for further investigations are the following.
\begin{question}
    Keeping in mind \autoref{thm:mainPGL} and \autoref{rem:q>n/log(n)}, 
    is $\coveringNumber(\PGL_n(\field_q))=n+2$ also true in the regime $C \frac{n}{\log(n)}\leq q\leq n^2$? 
\end{question}

\begin{conjecture}
    For every $n>1$ and algebraically closed field $\field$, we have $\coveringNumber(\PGL_n(\field))=n+2$. 
\end{conjecture}

\begin{problem}
    Determine $\packingNumber(G)$ and $\coveringNumber(G)$ for classical group if Lie type over algebraically closed, or finite, or arbitrary fields.
\end{problem}

In an irreducible reductive linear algebraic group $G$, all maximal tori are conjugate and their union is dense in $G$. 
The following natural direction is also connected to the Liebeck--Nikolov--Shalev conjecture \cite{LiebeckNikolovShalev} proved by \citeauthor{initiatingproof}
\cite{lifshitz2024completingproof}, \cite{initiatingproof}.
\begin{problem}
    Let $H$ be a closed irreducible subgroup in an irreducible linear algebraic group $G$, such that $\bigcup_{g\in G}H^g$ is dense in $G$. 
    Determine the covering and packing numbers when the notion of split maximal torus in \autoref{def:packingCOveringTiling} is replaced by conjugates of $H$.
\end{problem}

The bounds of \autoref{thm:mainBounds} are possibly sharper for simple groups.
\begin{conjecture}
    Then for some absolute constants $C_1, C_2$ so that for 
    every simple algebraic group $G$ (over an algebraically closed field), we have 
    \[C_1\frac{\dim(G)}{\rank(G)} 
    \leq \packingNumber(G)
    \leq \frac{\dim(G)}{\rank(G)}
    \leq \coveringNumber(G)\leq C_2 \frac{\dim(G)}{\rank(G)}.\]
\end{conjecture}

\autoref{thm:genericToriTransversal} and \autoref{lem:torusGrows} naturally raise the following conjecture, which may be true even for reductive groups.
\begin{conjecture}
    If $G$ is a simple linear algebraic group (over an algebraically closed field), and $V$ is a closed subset of $G$, 
    then there exists a maximal torus $T$ of $G$ with 
    \[\dim(VT)=\max\{\dim(V)+\dim(T),\dim(G)\}.\]
\end{conjecture}

\printbibliography

\end{document}